\documentclass[11pt,a4paper]{article}
\title{$L^2$ Properties of L\'{e}vy Generators on Compact Riemannian Manifolds}
\author{David Applebaum, Rosemary Shewell Brockway\\ School of Mathematics and Statistics,\\ University of
Sheffield,\\ Hicks Building, Hounsfield Road,\\ Sheffield,
England, S3 7RH\\ ~~~~~~~\\e-mail: D.Applebaum@sheffield.ac.uk, rjshewellbrockway1@sheffield.ac.uk}

\date{}
\usepackage[margin=1.99cm]{geometry}
\usepackage[shortlabels]{enumitem}
\usepackage{amsmath,amssymb,amsthm,mathtools,color,mathrsfs,tikz}
\numberwithin{equation}{section} 
\newtheorem{thm}{Theorem}[section]

\newtheorem{cor}[thm]{Corollary}
\newtheorem{lem}[thm]{Lemma}
\theoremstyle{definition}

\theoremstyle{example}

\theoremstyle{remark}
\newtheorem{rmk}[thm]{Remark}
\newtheorem*{ack}{Acknowledgement}
\usepackage[hide links]{hyperref}
\DeclareMathOperator{\Exp}{Exp}

\DeclareMathOperator{\tr}{tr}

\DeclareMathOperator{\Dom}{Dom}

\DeclareMathOperator{\ind}{\bf 1}

\begin{document}
\maketitle
\begin{abstract}
We consider isotropic L\'evy processes on a compact Riemannian manifold, obtained from an $\mathbb{R}^d$-valued L\'evy process through rolling without slipping. We prove that the Feller semigroups associated with these processes extend to strongly continuous contraction semigroups on $L^p$, for $1\leq p<\infty$. When $p=2$ we show that these semigroups are self-adjoint. If in addition, the motion has a non-trivial Brownian part, we prove that the generator has a discrete spectrum of eigenvalues and that the semigroup is trace-class.
\end{abstract}

\vspace{5pt}


\vspace{5pt}


\vspace{5pt}

\section{Introduction}
The investigation of L\'{e}vy processes (essentially processes with stationary and independent increments) has become a rich theme in modern probability theory. These processes have a deep and interesting structure, there are wide-ranging applications, and they have been found to be a useful class of driving noises for stochastic differential equations. Usually L\'{e}vy processes are considered as having Euclidean space (or even the real line) as their state space, although Banach and Hilbert space valued processes are important for stochastic partial differential equations (see e.g.~\cite{PZ}). The most intensively studied class of non-linear state spaces for L\'{e}vy processes have been Lie groups. Here the notion of increment is defined using the group law and inverse operation, so if $0 \leq s < t$, where the increment of the process $X$ in Euclidean space is $X(t) - X(s)$, on a Lie group we instead consider $X(s)^{-1}X(t)$. For a dedicated monograph treatment of such processes, see Liao \cite{Liao}.

Having considered L\'{e}vy processes on Lie groups, the next obvious step is to move to a Riemannian manifold. Indeed, arguably the most important L\'{e}vy process in Euclidean space is Brownian motion, and Brownian motion on manifolds is a very well developed theory in its own right (see e.g.~Elworthy \cite{elworthy} and Hsu \cite{hsu}). As its generator is the Laplace--Beltrami operator and transition density is the heat kernel, it is clear that this process is a natural expression of the geometry of the manifold. For development of a more general theory of L\'{e}vy processes, we are hampered by the fact that there is no obvious notion of increment on a manifold. When the space is a symmetric space, there is an alternative point of view pioneered by Gangolli \cite{Gang1, Gang2}, who used spherical functions to establish an analogue of the L\'{e}vy--Khintchine formula. Using the fact that every symmetric space is a homogeneous space $G/K$ of a Lie group $G$, quotiented by a compact subgroup $K$, Applebaum \cite{App1} was able to realise these processes as images of L\'{e}vy processes on $G$ through the natural map (see \cite{LiaoN} for more recent developments).

Motivated in part by the results of \cite{App1}, and also the Eels--Elworthy construction of Brownian motion on a manifold by projection from the frame bundle, Applebaum and Estrade (\cite{AppEstrade}) constructed a process that they called a ``L\'{e}vy process on a manifold'', and proved that it was a Feller--Markov process. The generic such process has the structure of a Brownian motion that is interlaced with jumps along geodesics, which are controlled by an isotropic L\'{e}vy measure. It is now more than twenty years since this paper was published, and to the authors' knowledge, there has been no further work carried out on such processes since that time, other than the ergodic theoretic analysis in Mohari \cite{Moh2}. In the current paper we hope to start the process of resurrecting this neglected area.

Since every L\'{e}vy process on a manifold is a Feller process, there is an associated semigroup on the space of continuous functions vanishing at infinity, and the generator is the sum of the Laplace--Beltrami operator and an integral superposition of jumps along geodesics. Our first new result is to show that the semigroup also preserves the $L^p$ space of the Riemannian volume measure. We are particularly interested in the case $p=2$ and here we show that the semigroup (and hence its generator) is self-adjoint. When there is a non-trivial Brownian motion component to the process, we find that the generator has a discrete spectrum of non-positive eigenvalues. Under the latter condition, we also prove that the semigroup is trace-class, and so the process has a transition density/heat kernel.

\section{L\'{e}vy Processes on Manifolds}
Let $M$ be a compact, connected, $d$-dimensional Riemannian manifold of dimension $d\geq 1$, with Riemannian measure $\mu$. Let $OM$ denote the corresponding bundle of orthonormal frames, a principle $O(d)$-bundle over $M$. The Levi-Civita connection induces a decomposition
\[T_rOM\cong H_rOM\oplus V_rOM, \hspace{15pt} \forall r\in OM\]
of each tangent space of $OM$ into vertical and horizontal subspaces, and, writing $\pi:OM\to M$ for the projection map, we have that for each $p\in M$ and $r\in OM_p$, $d\pi_r$ vanishes on $V_rOM$ and defines a linear isomorphism $H_rOM\cong T_pM$, for each $p\in M$ and $r\in OM_p$. Moreover, $OM$ may be given the structure of a Riemannian manifold in such a way that $d\pi_r:H_rOM\to T_pM$ is an isometric isomorphism\footnote{See Elworthy \cite{elworthy} Chapter III Section 4.}. The corresponding Riemannian measure $\tilde{\mu}$ is sometimes referred as Liouville measure. Observe that
\begin{equation}\label{eq:mu}
\mu=\tilde\mu\circ\pi^{-1}.
\end{equation}
One can easily show that $OM$ is compact, using the fact that both its base manifold and structure group are compact. As such, both $\mu$ and $\tilde\mu$ are finite measures.

The \emph{canonical horizontal fields} consist of a family of vector fields $\{H_x:x\in\mathbb{R}^d\}$ on $OM$, defined by
\[H_x(r)=r(x)^\ast, \hspace{15pt} \forall r\in OM,x\in\mathbb{R}^d,\]
where $^\ast$ denotes horizontal lift. Since $OM$ is compact, these vector fields are complete; we write $\Exp(tH_x)$ for the associated flows of diffeomorphisms. These flows are related to the Riemann exponential map by
\begin{equation}\label{eq:exp}
\pi(\Exp(H_x)(r))=\exp_pr(x), \hspace{15pt} \forall x\in\mathbb{R}^d, \; p\in M, \; r\in OM_p.
\end{equation}

For standard basis vectors $e_i$ of $\mathbb{R}^d$, we use the abbreviation $H_{e_i}=H_i$. The horizontal Laplacian
\[\Delta_H = \sum_{i=1}^dH_i^2,\]
on $OM$ is related to the Laplace--Beltrami operator $\Delta$ on $M$ by
\[\Delta f(p)=\Delta_H(f\circ\pi)(r), \hspace{15pt} \forall f\in C^\infty(M), \; p\in M, \; r\in OM_p;\]
for more details see Hsu \cite{hsu} Chapter 3.

Let $Y$ be an $\mathbb{R}^d$-valued L\'evy process. It is shown in Applebaum and Estrade \cite{AppEstrade} that the Marcus canonical SDE on $OM$
\begin{equation}\label{eq:marcus}
dR(t)= \sum_{j=1}^{d} H_j(R(t-))\diamond dY_j(t), \; t\geq 0; \hspace{15pt} R(0)=r \text{ (a.s.)},
\end{equation}
has a unique, c\`adl\`ag solution $R$ that is a Feller process. As in \cite{AppEstrade}, we will call processes obtained this way horizontal L\'evy processes. We also impose the assumption from \cite{AppEstrade} that the L\'evy process $Y$ is isotropic, in that its law is $O(d)$-invariant. By Corollary 2.4.22 on page 128 of \cite{dave}, the L\'{e}vy characteristics of $Y$ then take the form $(0,aI,\nu)$, where $a\geq 0$, and $\nu$ is $O(d)$-invariant. Then by Theorem 3.1 of \cite{AppEstrade}, the process $X=\pi(R)$ obtained by projection onto the base manifold is also a Feller process. The infinitesimal generators of $R$ and $X$ are
\begin{equation}\label{eq:L}
\mathscr{L}=\frac{1}{2}a\Delta_H+\mathscr{L}_J
\end{equation}
and
\begin{equation}\label{eq:A'}
\mathscr{A} = \frac{1}{2}a\Delta + \mathscr{A}_J,
\end{equation}
respectively, where the jump parts $\mathscr{L}_J$ and $\mathscr{A}_J$ are given by
\begin{equation}\label{eq:L+}
\mathscr{L}_Jf(r) = \int_{\mathbb{R}^d\setminus\{0\}}\left\{f(\Exp(H_x)(r))-f(r)-\ind_{|x|<1}H_xf(r)\right\}\nu(dx), \hspace{15pt} \forall f\in C^\infty(OM), \; r\in OM
\end{equation}
and
\[\mathscr{A}_Jf(p) = \int_{T_pM\setminus\{0\}}\left\{f(\exp_py)-f(p)-\ind_{|y|<1}yf(p)\right\}\nu_p(dx), \hspace{15pt} \forall f\in C^\infty(M), \; p\in M.\]
Here, the family of L\'evy measures $\{\nu_p:p\in M\}$ act on each tangent space $T_pM$, and are defined by $\nu_p=\nu\circ r^{-1}$ for any frame $r\in OM_p$. The two generators also satisfy
\[\mathscr{A}f(p)=\mathscr{L}(f\circ\pi)(r), \hspace{15pt} \forall f\in C^\infty(M), \;p\in M, \;r\in OM_p.\]

Observe that since $\nu$ is $O(d)$-invariant, the right hand side of (\ref{eq:L+}) is invariant under the change of variable $x\mapsto -x$, and hence for all $f\in C^\infty(OM)$ and $r\in OM$,
\begin{equation}\label{eq:L-}
\mathscr{L}_Jf(r) = \int_{\mathbb{R}^d\setminus\{0\}}\left\{f(\Exp(H_{-x})(r))-f(r)+\ind_{|x|<1}H_xf(r)\right\}\nu(dx),
\end{equation}
where we have used the fact that $H_{-x}f=-H_xf$. Summing (\ref{eq:L+}) and (\ref{eq:L-}) and dividing by two,
\begin{equation}\label{eq:L_J}
\mathscr{L}_Jf(r)=\frac{1}{2}\int_{\mathbb{R}^d\setminus\{0\}}\left\{f\big(\Exp(H_{x})(r)\big)-2f(r)+f(\Exp(H_{-x})(r))\right\}\nu(dx)
\end{equation}
for all $f\in C^\infty(OM)$ and $r\in OM$. It follows that
\begin{equation}\label{eq:A_J}
\mathscr{A}_Jf(p) = \frac{1}{2}\int_{T_pM\setminus\{0\}}\left[f(\exp_py)-2f(p)+f(\exp_p(-y))\right]\nu_p(dy)
\end{equation}
for each $f\in C^\infty(M)$ and $p\in M$. Note the analogous expression (5.4.16) in \cite{AppLieProb} for symmetric L\'evy motion on a Lie group.
\section{$L^p$ Properties of the Semigroup}
We now turn our attention to the Feller semigroups associated with $R$ and $X$. These consist of families of contraction operators $(S_t,t\geq 0)$ and $(T_t,t\geq 0)$ acting on the Banach spaces $C(OM)$ and $C(M)$ by
\[S_tf(r)=\mathbb{E}\big(f(R(t))|R(0)=r\big), \hspace{15pt} \forall f\in C(OM),\; r\in OM,\; t\geq 0,\]
and
\[T_tf(p)=\mathbb{E}\big(f(X(t))|X(0)=p\big), \hspace{15pt} \forall f\in C(M),\; p\in M,\; t\geq 0.\]
Since $X=\pi(R)$, it is immediate that
\begin{equation}\label{eq:Tt}
T_tf(p)=S_t(f\circ\pi)(r)
\end{equation}
for all $f\in C(M)$, $p=\pi(r)\in M$, and $t\geq 0$.

Any Riemannian manifold has a natural $L^p$ structure arising from its Riemannian measure, and so we may consider the spaces $L^p(OM)\coloneqq L^p(OM,\tilde{\mu},\mathbb{R})$ and $L^p(M)\coloneqq L^p(M,\mu,\mathbb{R})$ for $1\leq p\leq\infty$. For $p<\infty$, we prove that $(S_t,t\geq 0)$ and $(T_t, t\geq 0)$ extend to strongly continuous contraction semigroups on these spaces, and that they are self-adjoint when $p=2$.

We begin by considering the semigroup $(S_t,t\geq 0)$ associated with the horizontal process $R$; analogous results for $(T_t,t\geq 0)$ are then obtained by projection down onto $M$.
\begin{thm}\label{thm:extn}
For all $1\leq p<\infty$, $(S_t,t\geq 0)$ extends to a $C_0$-semigroup of contractions on $L^p(OM)$.
\end{thm}
\begin{proof}
Let $q_t(\cdot,\cdot)$ denote the transition measure of $(S_t,t\geq 0)$, so that
\[S_tf(r) = \int_{OM}f(u)q_t(r,du), \hspace{15pt} \forall t\geq 0, \; f\in C(OM), \; r\in OM.\]
The horizontal fields $H_x$ are divergence free for all $x\in\mathbb{R}^d$ (Proposition 4.1 of \cite{Moh1}), and so by Theorem 3.1 of \cite{AppKuInvMeas}, $\tilde{\mu}$ is invariant for $S_t$, in the sense that
\[\int_{OM}(S_tf)(r)\tilde{\mu}(dr) = \int_{OM}f(r)\tilde{\mu}(dr), \hspace{15pt} \forall t\geq 0, \; f\in C(OM).\]
Therefore, by Jensen's inequality,
\begin{align*}
\Vert S_tf\Vert_p^{p}=\int_{OM}\left|\int_{OM}f(u)q_t(r,du)\right|^p\tilde\mu(dr) &\leq \int_{OM}\int_{OM}|f(u)|^pq_t(r,du)\tilde\mu(dr) \\
&= \int_{OM}(S_t|f|^p)(r)\tilde\mu(dr) = \int_{OM}|f(r)|^p\tilde\mu(dr) = \Vert f\Vert_p^{p}
\end{align*}
for all $t\geq 0$ and $f\in C(OM)$. Each $S_t$ has domain $C(OM)$, which is a dense subspace of $L^p(OM)$ for all $1\leq p<\infty$. It follows that each $S_t$ extends to a unique contraction defined on the whole of $L^p(OM)$, which we also denote by $S_t$. By continuity, the semigroup property
\[S_tS_s=S_{t+s}, \hspace{15pt} \forall s,t\geq 0\]
continues to hold in this larger domain. It remains to prove strong continuity, i.e.~that
\begin{equation}\label{eq:strongcty}
\lim_{t\rightarrow 0}\Vert S_tf-f\Vert_p = 0
\end{equation}
for all $f\in L^p(OM)$. By density of $C(OM)$ in $L^p(OM)$, it is sufficient to verify this for $f\in C(OM)$. The map $t\mapsto S_t$ is strongly continuous in $C(OM)$, and so $\lim_{t\rightarrow 0}\Vert S_tf-f\Vert_\infty=0$ for all $f\in C(OM)$. Since $(OM,\mathcal{B}(OM),\tilde\mu)$ is a finite measure space,
\[\Vert S_tf-f\Vert_p\leq\tilde\mu(OM)^\frac{1}{p}\Vert S_tf-f\Vert_\infty,\]
for all $f\in C(OM)$. Equation (\ref{eq:strongcty}) now follows.	
\end{proof}
\begin{rmk}\label{rmk:strongcty}
The final part of the above proof applies more generally, in that if $X$ is a compact space equipped with a finite measure $m$, and if $(P_t,t\geq 0)$ is a Feller semigroup on $X$, then $(P_t,t\geq 0)$ is strongly continuous in $L^p(X, m)$.
\end{rmk}
Projection down onto the base manifold yields the following.
\begin{thm}\label{thm:selfadj}
For all $1\leq p<\infty$, $(T_t, t\geq 0)$ extends to a strongly-continuous semigroup of contractions on $L^p(M)$.
\end{thm}
\begin{proof}
Let $1\leq p<\infty$. By equations (\ref{eq:mu}) and (\ref{eq:Tt}), many of the conditions we must check follow from their analogues on the frame bundle. Indeed, for all $f\in C^\infty(M)$ and $t\geq 0$, we have
\[\Vert T_tf\Vert_{L^p(M)}^p = \int_M|T_tf(p)|^pd\mu = \int_{OM}|S_t(f\circ\pi)(r)|^pd\tilde\mu = \Vert S_t(f\circ\pi)\Vert_{L^p(OM)}^p,\]
and so, using the fact that $S_t$ is a contraction of $L^p(OM)$,
\begin{align*}
\Vert T_tf\Vert_{L^p(M)}=\Vert S_t(f\circ\pi)\Vert_{L^p(OM)}\leq\Vert f\circ\pi\Vert_{L^p(OM)}=\Vert f\Vert_{L^p(M)}.
\end{align*}
Hence $T_t$ extends to a contraction of $L^p(M)$ for all $t\geq 0$. It is clear by continuity that the semigroup property continues to hold on this larger domain, as does equation (\ref{eq:Tt}). Strong continuity follows by Remark \ref{rmk:strongcty}, or alternatively can be seen by the observation
\[\Vert T_tf-f\Vert_{L^p(M)} = \Vert S_t(f\circ\pi)-f\circ\pi\Vert_{L^p(OM)} \hspace{20pt} \forall f\in L^p(M).\]
Thus $(T_t, t\geq 0)$ extends to a contraction semigroup on $L^p(M)$ for all $1\leq p<\infty$.
\end{proof}

We continue to denote the generators of $(S_t,t\geq 0)$ and $(T_t,t\geq 0)$ by $\mathscr{L}$ and $\mathscr{A}$, respectively. Note that by Lemma 6.1.14 of \cite{davies2}, $\mathscr{L}$ and $\mathscr{A}$ are both closed operators on $L^p(OM)$.
\section{The Case $p=2$}
For the remainder of this paper we focus on the case $p=2$. Our aim in this section is to prove that the semigroups $(S_t,t\geq 0)$ and $(T_t,t\geq 0)$ are self-adjoint semigroups on $L^2(OM)$ and $L^2(M)$, respectively. By a standard result from semigroup theory, it will follow that $\mathscr{L}$ and $\mathscr{A}$ are self-adjoint linear operators.

Let us first impose the assumption that the L\'evy measure $\nu$ is finite. In this case, $\mathscr{A}_J$ is the generator of a compound Poisson process on $M$ (see \cite{AppEstrade}).
\begin{lem}\label{lem:selfadj}
If $\nu$ is finite, then $\mathscr{L}$ is a self-adjoint operator on $L^2(OM)$.
\end{lem}
\begin{proof}
Since $\nu$ is finite, $\mathscr{L}_J$ is a bounded linear operator on $L^2(OM)$, and so equation (\ref{eq:L_J}) extends by continuity to the whole of $L^2(OM)$. It follows that $\mathscr{L}$ is a bounded perturbation of the horizontal Laplacian, and so its domain is Dom$(\Delta_{H})$. Clearly $\mathscr{L}$ is symmetric on this domain.

Since $\mathscr{L}$ is a closed, symmetric operator, by Theorem X.1 on page 136 of Reed and Simon \cite{reed&simonII}, the spectrum $\sigma(\mathscr{L})$ of $\mathscr{L}$ is equal to one of the following:
\begin{enumerate}
\item The closed upper-half plane.
\item The closed lower-half plane.
\item The entire complex plane.
\item A subset of $\mathbb{R}$.\label{item:R}
\end{enumerate}
Moreover, $\mathscr{L}$ is self-adjoint if and only if Case \ref{item:R} holds. By Theorem 8.2.1 of \cite{davies2},
\begin{equation}\label{eq:specL}
\sigma(\mathscr{L})\subseteq(-\infty,0],
\end{equation}
from which we see that Case \ref{item:R} is the only option.
\end{proof}

We now drop the assumption that $\nu$ is finite.

\begin{thm}\label{thm:selfadj}
$(S_t, t\geq 0)$ and $(T_t,t\geq 0)$ are self-adjoint semigroups of operators on $L^2(OM)$ and $L^2(M)$, respectively.
\end{thm}
\begin{proof}
We will find it convenient to rewrite the process $R(t)$ (with initial condition $R(0) = r$ (a.s.)) as the action of a stochastic flow $\eta_t$ on the point $r$, as in \cite{AppEstrade}. Then as shown in Section 4 of \cite{AppEstrade}, there is a sequence $(\eta^{(n)}_t)$ of stochastic flows on $OM$ such that each $\eta^{(n)}_t$ is the flow of a horizontal L\'{e}vy process with finite L\'{e}vy measure, and
\[\lim_{n \rightarrow\infty}\eta^{(n)}_t(r) = \eta_t(r)~~~\mbox{(a.s.)},\]
for all $r \in OM$ and $t\geq 0$.

Let $(S_{t}^{(n)}, t \geq 0)$ be the transition semigroup corresponding to the flow $(\eta^{(n)}_t, t \geq 0)$, for each $n\in\mathbb{N}$. It is a standard result in semigroup theory that a semigroup of operators on a Hilbert space is self-adjoint if and only if its generator is self-adjoint\footnote{See for example Goldstein \cite{goldstein} page 31.}. By Lemma \ref{lem:selfadj}, $(S_t^{(n)},t\geq 0)$ is a self-adjoint semigroup on $L^2(OM)$, for all $n\in\mathbb{N}$.

By dominated convergence, for each $f \in C(OM)$ and $t\geq 0$, we have
\[\lim_{n \rightarrow\infty}\left\Vert S_{t}f - S_{t}^{(n)}f\right\Vert_{L^2(OM)}^{2} = \lim_{n \rightarrow \infty}\int_{OM}\left|\mathbb{E}\left(f\left(\eta_{t}(r)\right)-f\left(\eta_{t}^{(n)}(r)\right)\right)\right|^{2}\tilde{\mu}(dr) = 0.\]
Then by the density of $C(OM)$ in $L^{2}(OM)$, and a standard $\epsilon/3$ argument (using the fact that $S_{t}^{(n)}$ is an $L^{2}$-contraction), we deduce that for all $f\in L^{2}(OM)$,
\[\Vert S_{t}f\Vert_{L^2(OM)} = \lim_{n \rightarrow \infty}\left\Vert S_{t}^{(n)}f\right\Vert_{L^2(OM)}.\]
So $S_t$ is the strong limit of a sequence of bounded self-adjoint operators, and hence is itself self-adjoint. To see that $(T_t,t\geq 0)$ is also self-adjoint, let $t\geq 0$ and $f,g\in L^2(M)$, and observe that by (\ref{eq:mu}) and (\ref{eq:Tt}),
\[\langle T_tf,g\rangle_{L^2(M)} = \langle S_t(f\circ\pi),g\circ\pi\rangle_{L^2(OM)} = \langle f\circ\pi,S_t(g\circ\pi)\rangle_{L^2(OM)} = \langle f,T_tg\rangle_{L^2(M)}.\]
\end{proof}
By Theorem 4.6 of Davies \cite{davies1}, $-\mathscr{L}$ and $-\mathscr{A}$ are positive self-adjoint operators on $L^2(OM)$ and $L^2(M)$, respectively.
\section{Spectral Properties of the Generator}
For this final section, we restrict attention to the case in which $X$ has non-trivial Brownian part (that is, when $a>0$), and prove some spectral results that are already well-established for the case of Brownian motion and the Laplace--Beltrami operator $\Delta$. For example, it is well known that $\Delta$ has a discrete spectrum of eigenvalues. Each eigenspace is finite-dimensional, and the eigenvectors may be normalised so as to form an orthonormal basis of $L^2(M)$ (see for example Labl\'{e}e \cite{lablee} Theorem 4.3.1). Moreover, such an eigenbasis $(\psi_n)$ can be ordered so that the corresponding sequence of eigenvalues decreases to $-\infty$. For each $n\in\mathbb{N}$, write $-\mu_n$ for the eigenvalue associated with $\psi_n$, so that the real sequence $(\mu_n)$ satisfies
\begin{equation}\label{eq:LaplacianEigen}
0\leq\mu_1\leq\mu_2\leq\ldots\leq\mu_n\rightarrow\infty \text{ as } n\rightarrow\infty.
\end{equation}
We prove an analogous result for $\mathscr{A}$, using a generalisation of the approach used by Labl\'{e}e \cite{lablee}.
\begin{thm}\label{thm:DiscSpec}
Let $X$ be an isotropic L\'{e}vy process on $M$ with non-trivial Brownian part. Then its generator $\mathscr{A}$ has a discrete spectrum
\[\sigma(\mathscr{A}) = \{-\lambda_n:n\in\mathbb{N}\},\]
where $(\lambda_n)$ is a sequence of real numbers satisfying
\begin{equation}\label{eq:lambda}
0\leq \lambda_1\leq\lambda_2\leq\ldots\leq\lambda_n\rightarrow\infty \text{ as } n\rightarrow\infty.
\end{equation}
Moreover, each of the associated eigenspaces is finite-dimensional, and there is a corresponding sequence $(\phi_n)$ of eigenvectors that forms an orthonormal basis of $L^2(M)$.
\end{thm}
\begin{rmk}
We will generally assume that (\ref{eq:lambda}) is listed with multiplicity, so that for all $n\in\mathbb{N}$, $-\lambda_n$ is the eigenvalue associated with $\phi_n$.
\end{rmk}
\begin{proof}
Without loss of generality, assume that $a=2$ so that
\begin{equation}\label{eq:Delta+AJ}
\mathscr{A} = \Delta + \mathscr{A}_J,
\end{equation}
where $\mathscr{A}_J$ is given by (\ref{eq:A_J}). Both $\mathscr{A}$ and $\mathscr{A}_J$ are generators of self-adjoint contraction semigroups, and hence $-\mathscr{A}$ and $-\mathscr{A}_J$ are positive, self-adjoint operators\footnote{See Theorem 4.6 on page 99 of Davies \cite{davies1}.}. For $f,g\in\Dom\mathscr{A}$, define
\begin{equation}\label{eq:innerprodA}
\langle f,g\rangle_{\mathscr{A}} = \langle f,g\rangle_2 -\langle\mathscr{A}f,g\rangle_2.
\end{equation}
The operator $I-\mathscr{A}$ is also positive and self-adjoint, and so by Theorem 11 of \cite{bernau}, there is a unique positive self-adjoint operator $B$ such that $B^2=I-\mathscr{A}$. By (\ref{eq:specL}), $I-\mathscr{A}$ is invertible, and hence $B$ is injective. Moreover,
\[\langle f,g\rangle_{\mathscr{A}} = \langle Bf,Bg\rangle_2, \hspace{15pt} \forall f,g\in\Dom\mathscr{A},\]
from which it is easy\footnote{Bilinearity and symmetry are clear, and positive-definiteness is immediate by injectivity of $B$.} to see that $\langle\cdot,\cdot\rangle_{\mathscr{A}}$ defines an inner product on $\Dom\mathscr{A}$.

Let $V$ denote the completion of $C^\infty(M)$ with respect to $\langle\cdot,\cdot\rangle_{\mathscr{A}}$. This is a ``L\'{e}vy analogue'' of the Sobolev space $H^1(M)$ considered in Labl\'{e}e \cite{lablee} or Grigor'yan \cite{grigor'yan}, where the completion is instead taken with respect to the Sobolev inner product
\begin{equation}\label{eq:H1innerprod}
\langle f,g\rangle_{H^1} = \langle f,g\rangle_2 - \langle\Delta f,g\rangle_2, \hspace{15pt} \forall f,g\in C^\infty(M).
\end{equation}
In the case when $M = {\mathbb R}^{d}$, spaces of this type are discussed in Section 3.10 of Jacob \cite{jacob}, who refers to them as \emph{anisotropic Sobolev spaces}. By (\ref{eq:Delta+AJ}), we have
\[\langle f,g \rangle_{\mathscr{A}} = \langle f,g\rangle_{H^1} - \langle\mathscr{A}_J f,g\rangle_2 \hspace{15pt} \forall f,g\in C^\infty(M),\]
and since $-\mathscr{A}_J$ is a positive operator, it follows that $\Vert f\Vert_{\mathscr{A}} \geq \Vert f\Vert_{H^1}$ for all $f\in C^\infty(M)$. Similarly, (\ref{eq:H1innerprod}) implies $\Vert f\Vert_{H^1}\geq \Vert f\Vert_2$ for all $f\in C^\infty(M)$. Hence
\[V\subseteq H^1(M)\subseteq L^2(M),\]
and
\begin{equation}\label{eq:ineq}
\Vert f\Vert_2\leq\Vert f\Vert_{H^1}\leq \Vert f\Vert_{\mathscr{A}}, \hspace{15pt} \forall f\in V.
\end{equation}
In particular, the inclusion $V\hookrightarrow H^1(M)$ is bounded. By Rellich's theorem\footnote{See Labl\'{e}e \cite{lablee} page 68.}, the inclusion $H^1(M)\hookrightarrow L^2(M)$ is compact, and hence so is the inclusion $i:V\hookrightarrow L^2(M)$ (it is a composition of a compact operator with a bounded operator).

Let $f\in L^2(M)$ and consider $l\in V^\ast$ given by
\[l(g) = \langle f,g\rangle_2 \hspace{15pt} \forall g\in V.\]
For all $g\in V$ we have by the Cauchy-Schwarz inequality
\[|l(g)|\leq \Vert f\Vert_2\Vert g\Vert_2\leq\Vert f\Vert_2\Vert g\Vert_{\mathscr{A}}.\]
Hence
\begin{equation}\label{eq:cs}
\Vert l\Vert_{V^\ast}\leq \Vert f\Vert_2,
\end{equation}
where $\Vert\cdot\Vert_{V^\ast}$ denotes the norm of $V^\ast$. By the Riesz representation theorem, there is a unique $v_f\in V$ for which
\[\langle v_f,g\rangle_{\mathscr{A}} = l(g), \hspace{15pt} \forall g\in V.\]
Moreover,
\[\Vert v_f\Vert_{\mathscr{A}} = \sup_{g\in V\setminus\{0\}}\frac{|\langle v_f,g\rangle_{\mathscr{A}}|}{\Vert g\Vert_{\mathscr{A}}} = \Vert l\Vert_{V^\ast}.\]
Define $T:L^2(M)\to V$ by $Tf=v_f$ for all $f\in L^2(M)$. Then
\begin{equation}\label{eq:Tinnerprod}
\langle Tf,g\rangle_{\mathscr{A}} = \langle f,g\rangle_2 \hspace{15pt} \forall f\in L^2(M),g\in V,
\end{equation}
and $T$ is bounded, since by (\ref{eq:cs}),
\[\Vert Tf\Vert_{\mathscr{A}} = \Vert l\Vert_{V^\ast}\leq \Vert f\Vert_2,\]
for all $f\in L^2(M)$. By (\ref{eq:ineq}),
\[\Vert Tf\Vert_{\mathscr{A}}\leq\Vert f\Vert_{\mathscr{A}} \hspace{15pt} \forall f\in V,\]
and so $T|_V$ is a bounded operator on $V$. We also have
\[T|_V=T\circ i,\]
and, since $i$ is compact, so too is $T|_V$. By symmetry of inner products and equation (\ref{eq:Tinnerprod}), $T$ is self-adjoint. Equation (\ref{eq:Tinnerprod}) also implies that $T|_V$ is a positive operator, and that $0$ is not an eigenvalue of $T|_V$ (indeed, if $Tf=0$, then $\Vert f\Vert_2^2=\langle Tf,f\rangle_{\mathscr{A}}=0$).

By the Hilbert-Schmidt theorem\footnote{See Simon \cite{simon} Section 3.2.}, the spectrum of $T|_V$ consists of a sequence $(\alpha_n)$ of positive eigenvalues that decreases to $0$. Each eigenspace is finite-dimensional, and the corresponding eigenvectors can be normalised so as to form an orthonormal basis $(v_n)$ of $(V,\langle\cdot,\rangle_{\mathscr{A}})$.

In fact, it is easy to see from the definition of $\langle\cdot,\cdot\rangle_{\mathscr{A}}$ that
\begin{equation}\label{eq:T}
T=(I-\mathscr{A})^{-1},
\end{equation}
and hence the spectrum of $\mathscr{A}$ is just $\{1-\alpha_n^{-1}:n\in\mathbb{N}\}$, with corresponding eigenvectors still given by the $v_n$. Moreover, we may scale these eigenvectors so that they are $L^2$-orthonormal. Indeed, for each $n\in\mathbb{N}$, let
\[\phi_n = \frac{1}{\sqrt{\alpha_n}} v_n.\]
Then for all $m,n\in\mathbb{N}$,
\[\langle \phi_n,\phi_m\rangle_2 = \frac{1}{\sqrt{\alpha_n\alpha_m}}\langle Tv_n,v_m\rangle_{\mathscr{A}} = \sqrt{\frac{\alpha_n}{\alpha_m}}\langle v_n,v_m\rangle_{\mathscr{A}} = \delta_{m,n}.\]
By denseness of $V$ in  $L^2(M)$, the $\phi_n$ form an orthonormal basis of $L^2(M)$.

Finally, let $\lambda_n=\alpha_n^{-1}-1$ for each $n\in\mathbb{N}$. Then $(\lambda_n)$ satisfies equation (\ref{eq:lambda}), since $-\mathscr{A}$ is a positive operator, and $(\alpha_n)$ is a positive sequence that decreases to $0$.
\end{proof}

It is well-known that the heat semigroup $(K_t,t\geq 0)$ associated with Brownian motion on a compact manifold is trace-class, and possesses an integral kernel. The final two results of this section extend this to the L\'{e}vy semigroup $(T_t,t\geq 0)$, subject to the assumption that $a>0$.

\begin{thm}\label{thm:trace-class}
Let $X$ be an isotropic L\'{e}vy process on $M$ with non-trivial Brownian part. Then the transition semigroup operator $T_t$ is trace-class for all $t>0$.
\end{thm}
\begin{proof}
We again assume $a=2$, so that $\mathscr{A}$ has the form (\ref{eq:Delta+AJ}). The case for general $a>0$ is very similar.

Let $(\lambda_n$ and $(\phi_n)$ be as in the statement of Theorem \ref{thm:DiscSpec}, and let $(\mu_n)$ and $(\psi_n)$ be the analogous sequences for $\Delta$, so that $\psi_n$ is the $n^{\text{th}}$ eigenvector of $\Delta$, with associated eigenvalue $-\mu_n$.

Let $(K_t,t\geq 0)$ denote the heat semigroup associated with Brownian motion on $M$. This operator semigroup is known to possess many wonderful properties, including being trace-class. It follows that
\begin{equation}\label{eq:trKt}
\tr K_t = \sum_{n=1}^\infty\langle K_t\psi_n,\psi_n\rangle = \sum_{n=1}^\infty e^{-t\mu_n}<\infty
\end{equation}
for all $t>0$. As an element of $[0,\infty]$, the trace of each $T_t$ is given by
\begin{equation}\label{eq:trTt}
\tr T_t = \sum_{n=1}^\infty\langle T_t\phi_n,\phi_n\rangle = \sum_{n=1}^\infty e^{-t\lambda_n}.
\end{equation}
By the min-max principle for self-adjoint semibounded operators\footnote{Simon \cite{simon} page 666.}, we have for all $n\in\mathbb{N}$,
\[\lambda_n = -\sup_{f_1,\ldots,f_{n-1}\in C^\infty(M)}\left[\inf_{f\in\{f_1,\ldots,f_{n-1}\}^{\perp}, \;\Vert f\Vert=1}\langle\mathscr{A}f,f\rangle\right],\]
and
\[\mu_n = -\sup_{f_1,\ldots,f_{n-1}\in C^\infty(M)}\left[\inf_{f\in\{f_1,\ldots,f_{n-1}\}^{\perp}, \;\Vert f\Vert=1}\langle\Delta f,f\rangle\right].\]
As noted in the proof of Theorem \ref{thm:DiscSpec}, for all $f\in C^\infty(M)$,
\[-\langle\mathscr{A}f,f\rangle \geq -\langle\Delta f,f\rangle \geq 0,\]
and hence $\lambda_n\geq\mu_n$ for all $n\in\mathbb{N}$. But then $e^{-t\lambda_n}\leq e^{-t\mu_n}$ for all $t>0$ and $n\in\mathbb{N}$, and so, comparing (\ref{eq:trTt}) with (\ref{eq:trKt}),
\[\tr T_t<\tr K_t<\infty\]
for all $t>0$.
\end{proof}
By Lemma 7.2.1 of Davies \cite{davies2}, we immediately obtain the following.
\begin{cor}\label{cor:kernel}
Let $X$ be an isotropic L\'{e}vy process on $M$ with non-trivial Brownian part. Then its semigroup $(T_t,t\geq 0)$ has a square-integrable kernel. That is, for all $t > 0$ there is a map $p_t\in L^2(M\times M)$ such that
\[T_tf(x) = \int_Mf(y)p_t(x,y)\mu(dy)\]
for all $f\in L^2(M)$ and $x\in M$. Moreover, we have the following $L^2$-convergent expansion:
\[p_t(x,y) = \sum_{n=1}^\infty e^{-\lambda_nt}\phi_n(x)\phi_n(y), \hspace{15pt} \forall x,y\in M,\; t\geq 0.\]
\end{cor}

It is natural to enquire as to whether similar results hold in the pure jump case when $a=0$ (perhaps under some further condition on $\nu$)? When $M$ is a compact symmetric space, we can find an orthonormal basis of eigenfunctions that are common to the $L^{2}$--semigroups associated with all isotropic L\'{e}vy processes (see the results in section 5 of \cite{AT0}) . The key tool here, which enables a precise description of the spectrum of eigenvalues, is Gangolli's L\'{e}vy--Khinchine formula \cite{Gang1}. In the general case, as considered here, such methods are not available, and we are unable to make further progress at the present time.

\begin{ack}
We thank the referee for very helpful comments on an earlier version of this paper.
\end{ack}
\bibliographystyle{plain}
\bibliography{Levysgrps}
\end{document}